\documentclass[12pt]{article}

\usepackage{amssymb,amsmath,amsthm}
\usepackage[shortlabels]{enumitem}
\usepackage{xcolor}

\newcommand{\Q}{\mathbb{Q}}

\newcommand{\Gbar}{\overline{G}}

\newcommand{\tst}{\textstyle}

\newcommand{\Lrightarrow}{\hbox to1cm{\rightarrowfill}}
\newcommand{\Ldownarrow}{\bigg\downarrow}

\newcommand{\Z}{\mathbb{Z}}
\newcommand{\F}{\mathbb{F}}
\newcommand{\FE}{\mathrm{FE}}
\newcommand{\U}{\mathcal{U}}
\newcommand{\kbar}{\overline{k}}
\newcommand{\Kbar}{\overline{K}}
\newcommand{\Lbar}{\overline{L}}
\newcommand{\Mbar}{\overline{M}}
\newcommand{\Pbar}{\overline{P}}
\newcommand{\psibar}{\overline{\psi}}

\newcommand{\OO}{\mathcal{O}}
\DeclareMathOperator{\ch}{char}
\DeclareMathOperator{\e}{e}
\DeclareMathOperator{\Gal}{Gal}
\DeclareMathOperator{\Hom}{Hom}
\DeclareMathOperator{\Aut}{Aut}
\DeclareMathOperator{\rank}{rank}
\DeclareMathOperator{\W}{W}
\DeclareMathOperator{\Frac}{Frac}
\newcommand{\Spec}{{\rm Spec}}
\newcommand{\Gt}{\widetilde{G}}
\newcommand{\At}{\widetilde{A}}

\newtheorem{theorem}{Theorem}
\newtheorem{lemma}[theorem]{Lemma}
\newtheorem{prop}[theorem]{Proposition}
\newtheorem{cor}[theorem]{Corollary}

\theoremstyle{definition}

\newtheorem{definition}[theorem]{Definition}

\numberwithin{equation}{section}
\numberwithin{theorem}{section}

\title{Hasse-Arf property and abelian extensions for local fields with imperfect residue fields}
\date{\today}

\begin{document}
\author{{Taichi Inoue}}
\maketitle

\begin{abstract}
For a finite totally ramified extension $L$ of a complete discrete valuation field $K$ with the perfect residue field of characteristic $p>0$, it is known that $L/K$ is an abelian extension if the upper ramification breaks are integers and if the wild inertia group is abelian. 
We prove a similar result without the assumption that the residue field is perfect. 
As an application, we prove a converse to the Hasse-Arf theorem for a complete discrete valuation field with the imperfect residue field. 
More precisely, for a complete discrete valuation field $K$ with the residue field $\Kbar$ of residue characteristic $p>2$ and a finite non-abelian Galois extension $L/K$ such that the Galois group of $L/K$ is equal to the inertia group $I$ of $L/K$, we construct a complete discrete valuation field $K'$ with the residue field $\Kbar$ and a finite Galois extension $L'/K'$ which has at least one non-integral upper ramification break and whose Galois group and inertia group are isomorphic to $I$. 
\end{abstract}

\section{Introduction}
Let $K$ be a complete discrete valuation field and let $L/K$ be a finite Galois extension.
We assume that the residue extension of $L/K$ is separable.
For a rational number $r\ge-1$, the lower ramification group $\Gal(L/K)_r$ and the (classical) upper ramification group $\Gal(L/K)^r$ are defined.
We say that $u\ge-1$ is an {\it upper ramification break} of $L/K$ if $\Gal(L/K)^{u+\epsilon}\subsetneq\Gal(L/K)^{u}$ for any $\epsilon>0$.

For a field $k$ of characteristic $p>0$, we put $\kappa(k)=\dim_{\F_p}k/\wp(k)$, where $\wp(X)=X^p-X$.
The following proposition is proved in \cite[Remark 1]{feshas}.

\begin{prop}[{\cite[Remark 1]{feshas}}]\label{HAP}
Let $K$ be a complete discrete valuation field with the perfect residue field $\Kbar$ of residue characteristic $p>0$.
We assume that $\kappa(\Kbar)\neq0$.
Let $L/K$ be a finite Galois extension, and let $M/K$ be the maximal tamely ramified subextension of $L/K$.
We assume that the Galois group $\Gal(L/K)$ of $L/K$ is equal to the inertia group $I(L/K)$ of $L/K$.
If $L/M$ is abelian and the upper ramification breaks of $L/K$ are all integers, then $L/K$ is abelian.
\end{prop}

In \cite[Remark 3.5]{abbes-saito}, the upper numbering ramification group $\Gal(L/K)^r$  for a rational number $r\ge0$ is defined without the assumption that the residue extension of $L/K$ is separable.
In order to avoid confusion, we call them {\it the non-log upper ramification groups} and denote them by $\Gal(L/K)_{\text{n-log}}^r$.
We say that $u\ge0$ is a {\it non-log upper ramification break} of $L/K$ if $\Gal(L/K)_{\text{n-log}}^{u+\epsilon}\subsetneq\Gal(L/K)_{\text{n-log}}^{u}$ for any $\epsilon>0$.
In this article, we generalize Proposition \ref{HAP} to the case admitting $\kappa(\Kbar)=0$ or $\Kbar$ to be imperfect.
More precisely, we prove the following theorem.

\begin{theorem}[Theorem \ref{imperfect}]\label{mainth}
Let $K$ be a complete discrete valuation field with the residue field $\Kbar$, which is not necessarily perfect, of residue characteristic $p>0$.
Let $L/K$ be a finite Galois extension, and let $M/K$ be the maximal tamely ramified subextension of $L/K$.
We assume that the Galois group $\Gal(L/K)$ of $L/K$ is equal to the inertia group $I(L/K)$ of $L/K$.
If $L/M$ is abelian and the non-log upper ramification breaks of $L/K$ are all integers, then $L/K$ is abelian.
\end{theorem}

In Proposition \ref{HAP} and Theorem \ref{mainth}, the assumption that the Galois group $\Gal(L/K)$ of $L/K$ is equal to the inertia group $I(L/K)$ of $L/K$ is equivalent to that the residue extension of $L/K$ is purely inseparable.
Namely, in Proposition \ref{HAP}, $L/K$ is assumed to be totally ramified, since $\Kbar$ is perfect.
In Theorem \ref{mainth}, $L/K$ is not necessarily totally ramified.

We prove Theorem \ref{mainth} by using a tangentially dominant morphism defined in \cite[Definition 1.1.6]{graded}.
In Lemma \ref{tan-dom}, it follows from \cite[PROPOSITION 1.1.10]{graded} that there exists a complete discrete valuation field $K'$ with the perfect residue field of residue characteristic $p$ and there exists a tangentially dominant morphism $\OO_K\to \OO_{K'}$ of discrete valuation rings.
Then we can reduce the proof to the case the residue field $\Kbar$ is perfect by using \cite[COROLLARY 4.2.6]{graded}.

\vspace{\baselineskip}
In \cite{converse}, a converse to the Hasse-Arf theorem for a complete discrete valuation field with the perfect residue field is proved by using Proposition \ref{HAP}.
More precisely, the following theorem is proved in \cite[Remark 6.4]{converse} using Proposition \ref{HAP}.
\begin{theorem}[{\cite[Theorem 6.3]{converse}}]\label{converse-hasse}
Let $K$ be a complete discrete valuation field with the perfect residue field $\Kbar$ of residue characteristic $p>2$ and let $L/K$ be a finite Galois extension. 
We assume that the Galois group $G=\Gal(L/K)$ of $L/K$ is non-abelian and that $G$ is equal to the inertia group $I(L/K)$ of $L/K$.
Then there exist a complete discrete valuation field $K'$ with the residue field $\Kbar$ and a $G$-extension $L'/K'$ such that the Galois group $G=\Gal(L'/K')$ is equal to the inertia group $I(L'/K')$ of $L'/K'$, and that $L'/K'$ has a non-integral upper ramification break.
\end{theorem}

As an application of Theorem \ref{mainth}, we prove a similar theorem without the assumption that the residue field is perfect.

\begin{theorem}[Theorem \ref{converse2}]\label{converse-hasse-imp}
Let $K$ be a complete discrete valuation field with the residue field $\Kbar$ of residue characteristic $p>2$ and let $L/K$ be a finite Galois extension. 
We assume that the Galois group $G=\Gal(L/K)$ of $L/K$ is non-abelian and that $G$ is equal to the inertia group $I(L/K)$ of $L/K$.
Then there exist a complete discrete valuation field $K'$ with the residue field $\Kbar$ and a $G$-extension $L'/K'$ such that the Galois group $G=\Gal(L'/K')$ is equal to the inertia group $I(L'/K')$ of $L'/K'$, and that $L'/K'$ has a non-integral non-log upper ramification break.
\end{theorem}

\cite[Theorem 6.5]{converse} shows that there exists a $G$-extension of fields of characteristic $0$ which satisfies the conclusion of Theorem \ref{converse-hasse}.
We can prove a similar theorem without the assumption that the residue field is perfect.

\begin{theorem}[Theorem \ref{converse3}]\label{converse-hasse-imp-mix}
Let $K$ be a complete discrete valuation field with the residue field $\Kbar$ of residue characteristic $p>2$ and let $L/K$ be a finite Galois extension. 
We assume that the Galois group $G=\Gal(L/K)$ of $L/K$ is non-abelian and that $G$ is equal to the inertia group $I(L/K)$ of $L/K$.
Then there exist a complete discrete valuation field $K'$ of characteristic $0$ with the residue field $\Kbar$ and a $G$-extension $L'/K'$ such that the Galois group $G=\Gal(L'/K')$ is equal to the inertia group $I(L'/K')$ of $L'/K'$, and that $L'/K'$ has a non-integral non-log upper ramification break.
\end{theorem}

In Section \ref{gene1}, we generalize Proposition \ref{HAP} to the case $\kappa(\Kbar)\ge0$.
In Section \ref{gene2}, we prove Theorem \ref{mainth}.
In Section \ref{apply}, we prove Theorem \ref{converse-hasse-imp} and Theorem \ref{converse-hasse-imp-mix}.
In Section \ref{lemma}, we prove lemmas used in Section \ref{apply}.

\vspace{\baselineskip}
\noindent{\it Acknowledgements.}\
The author would like to thank his supervisor Yuri Yatagawa for her supports and advices on his study.

\vspace{\baselineskip}
\noindent{\it Notation.}\
For a discrete valuation field $K$, we denote by $\OO_K$, $\Kbar$, and $v_K$, respectively, the valuation ring, the residue field, and the valuation normalized so that $v_K(K^{\times})=\Z$.
For an extension $L/K$ of discrete valuation fields, we denote the ramification index of $L/K$ by $\e(L/K)$.
If $L/K$ is a finite Galois extension, we denote by $\Gal(L/K)$, $I(L/K)$, and $P(L/K)$, respectively, the Galois group, the inertia group, and the wild inertia group of $L/K$.
For a positive integer $m$, we put $C_m=\Z/m\Z$.

\section{Generalization to the case $\kappa$ is equal to zero} \label{gene1}
For a field $k$ of characteristic $p>0$, we put $\kappa(k)=\dim_{\F_p}k/\wp(k)$, where $\wp(X)=X^p-X$.
The following proposition is proved in \cite[Remark 1]{feshas}.

\begin{prop}[{\cite[Remark 1]{feshas}}]\label{HAP2}
Let $K$ be a complete discrete valuation field with the perfect residue field $\Kbar$ of residue characteristic $p>0$.
We assume that $\kappa(\Kbar)\neq0$.
Let $L/K$ be a finite Galois extension, and let $M/K$ be the maximal tamely ramified subextension of $L/K$.
We assume that the Galois group $\Gal(L/K)$ of $L/K$ is equal to the inertia group $I(L/K)$ of $L/K$.
If $L/M$ is abelian and the upper ramification breaks of $L/K$ are all integers, then $L/K$ is abelian.
\end{prop}

In \cite{feshas}, it is stated that Proposition \ref{HAP2} in the case $\kappa(\Kbar)=0$ may be treated without any detailed proof.
In this section, we give a detailed proof for Proposition \ref{HAP2} in the case $\kappa(\Kbar)\ge0$. 

\begin{lemma} \label{alg-ind}
Let $k$ be a field of characteristic $p>0$ such that $\kappa(k)\neq0$, and let $\kbar$ be an algebraic closure of $k$. Let a subset $S$ of $\kbar$ be algebraically independent over $k$. Then the field $k(S)$ generated by $S$ over $k$ satisfies $\kappa(k(S))\neq0$.
\end{lemma}
\begin{proof}
Let $a\in k\setminus\wp(k)$, and let $\alpha\in \kbar$ be a root of $X^p-X-a$. 
We assume that $\alpha\in k(S)$. 
Then $\alpha$ can be written as\[\alpha=\frac{f(s_1, s_2, \dots, s_n)}{g(s_1, s_2, \dots, s_n)}\] for $s_1, s_2, \dots, s_n\in S$ and $f,g\in k[X_1, X_2, \dots, X_n]\setminus\{0\}$. 
It follows from $\alpha^p-\alpha-a=0$ that \[f(s_1, s_2, \dots, s_n)^p-f(s_1, s_2, \dots, s_n)g(s_1, s_2, \dots, s_n)^{p-1}-ag(s_1, s_2, \dots, s_n)^p=0.\] 
Since $S$ is algebraically independent over $k$, we get $f^p-fg^{p-1}-ag^p=0$.
If $g\in k^{\times}$, then we have $f\in k^{\times}$ and $\alpha\in k^{\times}$, since $g^p=a^{-1}f(f^{p-1}-g^{p-1})$. 
This contradicts the assumption that $a\not\in \wp(k)$.
If $g\not\in k^{\times}$, then $g$ can be written as a finite product of prime elements $g_1, g_2, \dots, g_m\in k[X_1, X_2, \dots, X_n]$, since $k[X_1, X_2, \dots, X_n]$ is a unique factorization domain. 
We can assume that $f\not\in (g_1)$, and then $f^p\not\in (g_1)$.
However, we have $f^p=g^{p-1}(f-ag)\in(g_1)$.
This is a contradiction.
Therefore, we have $\alpha\not\in k(S)$ and $a\in k(S)\setminus\wp(k(S))$.
\end{proof}

\begin{lemma} \label{per-clo}
Let $k$ be a field of characteristic $p>0$ such that $\kappa(k)\neq0$. Then the perfect closure $k'=\bigcup_{n\ge0} k^{p^{-n}}$ of $k$ satisfies $\kappa(k')\neq0$.
\end{lemma}
\begin{proof}
Let $\kbar$ be an algebraic closure of $k$.
Let $a\in k\setminus\wp(k)$, and let $\alpha\in \kbar$ be a root of $X^p-X-a$. 
Since $\alpha$ is separable over $k$ and $k'$ is purely inseparable over $k$, we have $\alpha\not\in k'$.
Hence we have $a\in k'\setminus\wp(k')$.
\end{proof}

\begin{lemma} \label{finite}
Let $k$ be a field of characteristic $p>0$ such that $\kappa(k)\neq0$. Let $k'$ be a finite extension of $k$. Then $k'$ satisfies $\kappa(k')\neq0$.
\end{lemma}
\begin{proof}
Let $m\ge1$ be an integer such that $p^m>[k':k]$.
Since $\kappa(k)\neq0$, there exists a cyclic extension $k''/k$ of degree $p^m$ by \cite[Satz 13]{witZ}.
Then the composite field $k'k''$ is a non-trivial cyclic extension over $k'$ of degree a power of $p$, and Galois group $\Gal(k'k''/k')$ has a normal subgroup of index $p$.
Therefore, there exists a cyclic extension over $k'$ of degree $p$.
\end{proof}

Let $K$ and $L$ be complete discrete valuation fields.
If $L$ is an extension of $K$ and $v_L(x)=\e(L/K)v_K(x)$ for any $x\in K^{\times}$, we say that $L/K$ is an {\it extension of complete discrete valuation fields}.

\begin{lemma} \label{reduce1}
Let $K'$ be a complete discrete valuation field of residue characteristic $p>0$, and let $K/K'$ be a extension of complete discrete valuation fields such that $\e(K/K')=1$.
Let $L'/K'$ be a finite totally ramified extension, and put $L=KL'$.
\begin{enumerate}[$(1)$]
\item $L/K$ is a totally ramified extension such that $[L:K]=[L':K']$ and $\e(L/L')=1$.
\item We assume that $L'/K'$ is a Galois extension.
Then $L/K$ is a Galois extension whose Galois group $\Gal(L/K)$ is isomorphic to the Galois group $\Gal(L'/K')$ of $L'/K'$.
Moreover, $u\ge 0$ is an upper ramification break of $L'/K'$ if and only if $u$ is an upper ramification break of $L/K$.
\end{enumerate}
\end{lemma}
\begin{proof}
$(1)$ Since $L'/K'$ is totally ramified, $L'$ is generated by a uniformizer $\pi_{L'}$ for $L'$ over $K'$ by \cite[$\mathrm{I}$ Proposition 18]{ser}. 
Then $L$ is generated by $\pi_{L'}$ over $K$.
Moreover, the minimal polynomial $f$ of $\pi_{L'}$ over $K'$ is an Eisenstein polynomial.
Since $\e(K/K')=1$, $f$ is Eisenstein over $K$ and hence irreducible over $K$.
Therefore, $L/K$ is totally ramified and $[L:K]=[L':K']$ by \cite[$\mathrm{I}$ Proposition 17]{ser}.
Since $L/K$ and $L'/K'$ are totally ramified and $\e(K/K')=1$, we have 
\begin{align*}
\e(L/L')[L':K']
&=\e(L/K')=\e(K/K')[L:K]\\
&=[L:K]=[L':K']
\end{align*}
and hence $\e(L/L')=1$.

$(2)$
Since $L'/K'$ is totally ramified and $\e(K/K')=1$, we have $L'\cap K=K'$.
Hence $L/K$ is a Galois extension such that $\Gal(L/K)\cong \Gal(L'/K')$.
Let $\pi_{L'}$ be a uniformizer for $L'$.
Since $\OO_L=\OO_K[\pi_{L'}]$ and $\OO_{L'}=\OO_{K'}[\pi_{L'}]$ with $e(L/L')=1$, we have
\begin{align*}
\Gal(L/K)_a
&=\{\sigma\in \Gal(L/K):v_L(\sigma(\pi_{L'})-\pi_{L'})\ge a+1\}\\
&\cong\{\sigma\in \Gal(L'/K'):v_{L'}(\sigma(\pi_{L'})-\pi_{L'})\ge a+1\}\\
&=\Gal(L'/K')_a
\end{align*}
for any $a\ge-1$.
Therefore, the assertion follows.
\end{proof}

\begin{cor} \label{subex}
In Lemma \ref{reduce1}, let $M'$ be a subextension of $L'/K'$ and put $M=KM'$.
\begin{enumerate}[$(1)$]
\item If $M'/K'$ is the maximal tamely ramified subextension of $L'/K'$, then $M/K$ is the maximal tamely ramified subextension of $L/K$.
\item We assume that $L'/M'$ is a Galois extension. 
Then $L/M$ is a Galois extension whose Galois group $\Gal(L/M)$ is isomorphic to the Galois group $\Gal(L'/M')$ of $L'/M'$.
Moreover, $u\ge 0$ is an upper ramification break of $L'/M'$ if and only if $u$ is an upper ramification break of $L/M$.
\end{enumerate}
\end{cor}
\begin{proof}
$(1)$ By Lemma \ref{reduce1} $(1)$, $L/K$ is totally ramified, and we get $\e(L/K)=\e(L'/K')$.
If $\e(L'/K')=mp^n$ with $p\nmid m$, then we have $\e(L'/M')=p^n$ and $\e(M'/K')=m$.
Applying Lemma \ref{reduce1} $(1)$ to $M'/K'$ and $K/K'$, we have $\e(M/K)=\e(M'/K')=m$, and hence $\e(L/M)=\e(L'/M')=p^n$.
Therefore, $M/K$ is the maximal tamely ramified subextension.

$(2)$ Applying Lemma \ref{reduce1} $(1)$ to $M'/K'$ and $K/K'$, we have $\e(M/M')=1$.
Thus we can apply Lemma \ref{reduce1} $(2)$ to $L'/M'$ and $M/M'$.
Therefore, the assertion follows.
\end{proof}

\begin{prop}[{cf.\ \cite[Remark 1]{feshas}}]\label{kappa}
Let $K$ be a complete discrete valuation field with the perfect residue field $\Kbar$ of residue characteristic $p>0$.
Let $L/K$ be a finite Galois extension, and let $M/K$ be the maximal tamely ramified subextension of $L/K$.
We assume that the Galois group $\Gal(L/K)$ of $L/K$ is equal to the inertia group $I(L/K)$ of $L/K$.
If $L/M$ is abelian and if the upper ramification breaks of $L/K$ are all integers, then $L/K$ is abelian.
\end{prop}
\begin{proof}
If $\kappa(\Kbar)\neq0$, the assertion follows from Proposition \ref{HAP2}.
We assume that $\kappa(\Kbar)=0$.
Let $S$ be a transcendence basis of $\Kbar$ over $\F_p$, and set $k'=\F_p(S)$ and $k=\bigcup_{n\ge0} k'^{p^{-n}}$.
By Lemma \ref{alg-ind} and Lemma \ref{per-clo}, we have $\kappa(k)\neq0$ and $\Kbar/k$ is algebraic.

It suffices to show that there exist extensions $K/F$ and $F/F''$ of complete discrete valuation fields, such that $K/F$ is a finite totally ramified extension, the residue field of $F''$ is $k$, and that $\e(F/F'')=1$.
Actually, we can show that $\Gal(L/K)$ is abelian by using such extensions as follows:

Since $L/K$ and $K/F$ are finite totally ramified extensions, so is $L/F$.
Hence $\OO_L$ is generated by a uniformizer $\pi_L$ over $\OO_F$ and $\OO_K$ is generated by a uniformizer $\pi_K$ over $\OO_F$ by \cite[$\mathrm{I}$ Proposition 18]{ser}. 
Moreover, the minimal polynomial $f$ of $\pi_L$ over $F$ is an Eisenstein polynomial.
Since $L/K$ is a Galois extension, the conjugate elements $\alpha_i$ of $\pi_L$ over $K$ belong to $\OO_L=\OO_F[\pi_L]$, and $\alpha_i$ can be written as $\alpha_i=g_i(\pi_L)$ for $g_i\in \OO_F[X]$.
Let $C\subset\OO_F$ be the finite set of the coefficients of $f$ and $g_i$.
Then $F''(C)\subset F$ is a discrete valuation field with the valuation $v_F|_{F''(C)}$.
Let $F'\subset F$ be the completion of $F''(C)$.
Then the residue extension $\overline{F'}/k$ of $F'/F''$ is finite, since $\overline{F'}=\overline{F''(C)}$ is generated by the images of elements of $C$ in $\overline{F}$ over $k$.
Since $k$ is perfect, so is $\overline{F'}$.
By Lemma \ref{finite}, $\overline{F'}$ satisfies $\kappa(\overline{F'})\neq0$.
Put $L'=F'(\pi_L)$ and $K'=F'(\pi_K)$. 
Then we have $K=FK'$ and $L=FL'=KL'$.
Since $\e(F/F')=1$, $f$ is Eisenstein over $F'$ and hence irreducible over $F'$.
Thus $L'/F'$ is a finite totally ramified extension by \cite[$\mathrm{I}$ Proposition 17]{ser}.
By Lemma \ref{reduce1} $(1)$, we have $e(K/K')=1$.
Hence we get $K\cap L'=K'$, since $L'/K'$ is a finite totally ramified extension.
Since $\alpha_i=g_i(\pi_L)\in F'(\pi_L)=L'$, the coefficients of the minimal polynomial of $\pi_L$ over $K$ belong to $K\cap L'=K'$.
Hence, the conjugate elements of $\pi_L$ over $K'$ belong to $L'$ and $L'/K'$ is a Galois extension.
Let $M'/K'$ be the maximal tamely ramified subextension of $L'/K'$. 
Applying Corollary \ref{subex} to a Galois extension $L'/F'$, a extension $F/F'$ and subextensions $K'/F'$ and $M'/F'$ of $L'/F'$, we get $\Gal(L/K)\cong\Gal(L'/K')$, $M'=KM'$, and $\Gal(L/M)\cong\Gal(L'/M')$, and the upper ramification breaks of $L'/K'$ are all integers. 
By Proposition \ref{HAP2}, the Galois group $\Gal(L/K)\cong\Gal(L'/K')$ is abelian, which is what we desired.

We construct the desired fields $F$ and $F''$.

$\mathrm{i})$ Consider the case $\ch(K)=0$.
By \cite[$\mathrm{II}$ Theorem 3]{ser}, there exists a complete discrete valuation ring $W(\Kbar)$ (resp. $W(k)$) of characteristic $0$ with the residue field $\Kbar$ (resp. $k$) which is absolutely unramified.
Put $F=\Frac(\W(\Kbar))$ and $F''=\Frac(\W(k))$.
Then there exists an injection $F\to K$ such that $K/F$ is a finite totally ramified extension by \cite[$\mathrm{II}$ Theorem 4]{ser}.
There exists an injection $F''\to F$ such that $\e(F/F'')=1$ by \cite[$\mathrm{II}$ Proposition 10]{ser}.

$\mathrm{ii})$ Consider the case $\ch(K)=p$.
Put $F=K$.
By \cite[$\mathrm{II}$ Theorem 2]{ser}, $K$ is isomorphic to $\Kbar((T))$. 
Let $F''\subset K$ be the subfield  corresponding to  $k((T))\subset \Kbar((T))$.
Then $F''$ is a complete discrete valuation field with the residue field $k$ and $\e(K/F'')=1$.
\end{proof}

\section{Generalization to the case the residue field is imperfect} \label{gene2}
Let $K$ be a complete discrete valuation field with the residue field $\Kbar$, which is not necessarily perfect, of residue characteristic $p>0$ , and let $L/K$ be a finite Galois extension.
The upper numbering ramification groups $\Gal(L/K)^r\subset\Gal(L/K)$ is defined in \cite[Remark 3.5]{abbes-saito}.
In order to avoid confusion, we call them the {\it non-log upper ramification groups} and denote them by $\Gal(L/K)_{\text{n-log}}^r$.
We say that $u\ge0$ is a {\it non-log upper ramification break} of $L/K$ if $\Gal(L/K)_{\text{n-log}}^{u+\epsilon}\subsetneq\Gal(L/K)_{\text{n-log}}^{u}$ for any $\epsilon>0$.
If a subextension $M/K$ of $L/K$ is a Galois extension, it follows that $\Gal(M/K)_{\text{n-log}}^r=\Gal(L/K)_{\text{n-log}}^r\Gal(L/M)/\Gal(L/M)$.
In particular, if $u\ge0$ is a non-log upper ramification break of $M/K$, then $u$ is a non-log upper ramification break of $L/K$.
By \cite[Proposition 3.7]{abbes-saito}, the inertia group $I(L/K)$ of $L/K$ is equal to $\Gal(L/K)_{\text{n-log}}^1$, and the wild inertia group $P(L/K)$ of $L/K$ is equal to $\bigcup_{r>1}\Gal(L/K)_{\text{n-log}}^r$.
In addition, if the residue extension of $L/K$ is separable, $\Gal(L/K)_{\text{n-log}}^r$ is isomorphic to the classical upper ramification group $\Gal(L/K)^{r-1}$ for any $r\ge 0$.
In this section, we generalize Proposition \ref{kappa} to the case the residue field is not necessarily perfect.

\vspace{\baselineskip}
Let $K$ be a discrete valuation field.
If $\Kbar^{alg}$ is an algebraic closure of $\Kbar$, the cotangent complex $L_{\Kbar^{alg}/\OO_K}$ is defined in \cite[$\mathrm{II}$, 1.2.3]{illusie}.
The tangentially dominant morphism of discrete valuation rings are defined in \cite[Definition 1.1.6]{graded} as follows:

\begin{definition}
Let $K$ and $K'$ be discrete valuation fields with the residue field $\Kbar$ and $\overline{K'}$. 
We say that a morphism $\OO_K\to \OO_{K'}$ is a {\it morphism of discrete valuation rings} if it is faithfully flat, i.e., $\OO_K\to \OO_{K'}$ is flat and $\Spec\,\OO_{K'}\to\Spec\,\OO_K$ is surjective.
Let $\OO_K\to \OO_{K'}$ be a morphism of discrete valuation rings and $\Kbar^{alg}\to \overline{K'}^{alg}$ be a morphism of algebraic closures of residue fields extending $\Kbar\to \overline{K'}$.
We say that a morphism $\OO_K\to \OO_{K'}$ of discrete valuation rings is {\it tangentially dominant} if the the canonical morphism
\[S(H_1(L_{\Kbar^{alg}/\OO_K}))\to S(H_1(L_{\overline{K'}^{alg}/\OO_{K'}}))\]
of symmetric algebras is an injection.
\end{definition}

If $K$ and $K'$ are complete discrete valuation fields and $\OO_K\to \OO_{K'}$ is a morphism of discrete valuation rings, then $\OO_K\to \OO_{K'}$ is an injection and $K'/K$ is a extension of complete discrete valuation fields.

\begin{lemma}[{cf.\ \cite[PROPOSITION 1.1.10]{graded}}]\label{tan-dom}
Let $K$ be a complete discrete valuation field of residue characteristic $p>0$.
Then there exists a complete discrete valuation field $K'$ with the perfect residue field of residue characteristic $p$ and there exists a tangentially dominant morphism $\OO_K\to \OO_{K'}$ of discrete valuation rings. 
\end{lemma}
\begin{proof}
By \cite[PROPOSITION 1.1.10]{graded}, there exists a discrete valuation field $K''$ with the perfect residue field of residue characteristic $p$ and there exists a tangentially dominant morphism $\OO_K\to \OO_{K''}$ of discrete valuation rings. 
Let $K'$ be the completion of $K''$.
Then we have $\e(K'/K'')=1$ and $\overline{K'}=\overline{K''}$.
Thus the canonical morphism $\OO_{K''}\to \OO_{K'}$ is tangentially dominant by \cite[PROPOSITION 1.1.8  (b)]{graded}.
Therefore, the composition $\OO_K\to \OO_{K'}$ is tangentially dominant by the injectivity of the composition of two injections of symmetric algebras.
\end{proof}

\begin{lemma}[{cf.\ \cite[COROLLARY 4.2.6]{graded}}]\label{reduce2}
Let $K$ be a complete discrete valuation field of residue characteristic $p>0$.
Let $L/K$ be a finite Galois extension, and let $M/K$ the maximal tamely ramified subextension of $L/K$.
We assume that the Galois group $\Gal(L/K)$ of $L/K$ is equal to the inertia group $I(L/K)$ of $L/K$.
Let $K'$ be a complete discrete valuation field with the perfect residue field of residue characteristic $p>0$ and $\OO_K\to \OO_{K'}$ be a tangentially dominant morphism.
Let $L'=LK'$  and $M'=MK'$ be composite fields.
Then $L'/K'$ is a finite totally ramified Galois extension with the Galois group $\Gal(L'/K')$ isomorphic to $\Gal(L/K)$, and $M'/K'$ is the maximal tamely ramified subextension of $L'/K'$ with $\Gal(L'/M')$ isomorphic to $\Gal(L/M)$.
Moreover, $u\ge 1$ is a non-log upper ramification break of $L/K$ if and only if $u-1$ is an upper ramification break of $L'/K'$.
\end{lemma}
\begin{proof}
Since $L/K$ is a finite Galois extension, $L'/K'$ is a finite Galois extension.
By \cite[PROPOSITION 1.1.8 (b)]{graded}, we have $\e(K'/K)=1$.
Since $M/K$ is totally ramified and $\e(K'/K)=1$, $M'/K'$ is totally ramified and we have $\e(M'/K')=[M':K']=[M:K]=\e(M/K)$ by Lemma \ref{reduce1} (1).
Since $\e(M'/K')=\e(M/K)$ is indivisible by $p$, $M'/K'$ is tamely ramified.
Let $M''/K'$ be the maximal tamely ramified subextension of $L'/K'$.
By \cite[COROLLARY 4.2.6]{graded}, for any $r>1$, $\Gal(L/K)_{\text{n-log}}^r$ is isomorphic to $\Gal(L'/K')_{\text{n-log}}^r$.
Hence the wild inertia group $\Gal(L/M)$ of $L/K$ is isomorphic to the wild inertia group $\Gal(L'/M'')$ of $L'/K'$.
Since $M'\subset M''$, we have 
\[[L:K]\ge [L':K']\ge [L':M''][M':K']=[L:M][M:K]=[L:K].\]
Thus we get $M'=M''$ and the canonical injection $\Gal(L'/K')\to\Gal(L/K)$ is an isomorphism.
Since $L'/M''$ and $M'/K'$ are totally ramified, $L'/K'$ is totally ramified and we have $\Gal(L'/K')_{\text{n-log}}^1=I(L'/K')=\Gal(L'/K')$.
Since $\Gal(L/K)_{\text{n-log}}^1=I(L/K)=\Gal(L/K)$, we get $\Gal(L/K)_{\text{n-log}}^r\cong\Gal(L'/K')_{\text{n-log}}^r$ for any $r\ge 1$.
Moreover, the non-log upper ramification group $\Gal(L'/K')_{\text{n-log}}^r$ is isomorphic to the classical upper ramification group $\Gal(L'/K')^{r-1}$ for $r\ge1$.
Therefore, the assertion follows.
\end{proof}

\begin{theorem}[{cf.\ \cite[Remark 1]{feshas}}]\label{imperfect}
Let $K$ be a complete discrete valuation field of residue characteristic $p>0$.
Let $L/K$ be a finite Galois extension, and let $M/K$ be the maximal tamely ramified subextension of $L/K$.
We assume that the Galois group $\Gal(L/K)$ of $L/K$ is equal to the inertia group $I(L/K)$ of $L/K$.
If $L/M$ is abelian and the non-log upper ramification breaks of $L/K$ are all integers, then $L/K$ is abelian.
\end{theorem}
\begin{proof}
By Lemma \ref{tan-dom}, there exists a complete discrete valuation field $K'$ with the perfect residue field of residue characteristic $p$ and there exists a tangentially dominant morphism $\OO_K\to \OO_{K'}$ of discrete valuation rings. 
Let $L'=LK'$ be a composite field of $L$ and $K'$, and let $M'/K'$ be the maximal tamely ramified subextension of $L'/K'$.
By Lemma \ref{reduce2}, $L'/K'$ is a finite totally ramified Galois extension, and we have $\Gal(L/K)\cong\Gal(L'/K')$ and $\Gal(L/M)\cong\Gal(L'/M')$.
Moreover, the upper ramification breaks of $L'/K'$ are all integers.
By Proposition \ref{kappa}, the Galois group $\Gal(L/K)\cong\Gal(L'/K')$ is abelian.
\end{proof}


\section{Application: a converse to the Hasse-Arf theorem} \label{apply}
Let $K$ be a complete discrete valuation field with the residue field $\Kbar$ of residue characteristic $p>0$ and $L/K$ be a finite Galois extension.
The Hasse-Arf theorem states that if $G=\Gal(L/K)$ is abelian, then the upper ramification breaks of $L/K$ must be integers.
In this section, we consider a converse to the Hasse-Arf theorem.
If the residue field $\Kbar$ is perfect, the following theorem is proved in \cite[Theorem 6.3]{converse}.

\begin{theorem}[{\cite[Theorem 6.3]{converse}}]\label{converse1}
Let $K$ be a complete discrete valuation field with the perfect residue field $\Kbar$ of residue characteristic $p>2$ and let $L/K$ be a finite Galois extension. 
We assume that the Galois group $G=\Gal(L/K)$ of $L/K$ is non-abelian and that $G$ is equal to the inertia group $I(L/K)$ of $L/K$.
Then there exist a complete discrete valuation field $K'$ with the residue field $\Kbar$ and a $G$-extension $L'/K'$ (i.e., the Galois group is isomorphic to $G$) such that the Galois group $G=\Gal(L'/K')$ is equal to the inertia group $I(L'/K')$ of $L'/K'$, and that $L'/K'$ has a non-integral upper ramification break.
\end{theorem}

We prove a similar theorem without the assumption that the residue field is perfect.
In Theorem \ref{converse2} below, $L/K$ is not necessarily totally ramified, while $L/K$ is totally ramified in Theorem \ref{converse1}.

\begin{theorem}[{cf.\ \cite[Theorem 6.3]{converse}}]\label{converse2}
Let $K$ be a complete discrete valuation field with the residue field $\Kbar$ of residue characteristic $p>2$ and let $L/K$ be a finite Galois extension. 
We assume that the Galois group $G=\Gal(L/K)$ of $L/K$ is non-abelian and that $G$ is equal to the inertia group $I(L/K)$ of $L/K$.
Then there exist a complete discrete valuation field $K'$ with the residue field $\Kbar$ and a $G$-extension $L'/K'$ such that the Galois group $G=\Gal(L'/K')$ is equal to the inertia group $I(L'/K')$ of $L'/K'$, and that $L'/K'$ has a non-integral non-log upper ramification break.
\end{theorem}
\begin{proof}
Let $P$ be the wild inertia group of $L/K$ and put $m=(G:P)$.
Then we have $G/P\cong C_m=\Z/m\Z$, and $G\cong P\rtimes_{\psi}C_m$ for some homomorphism $\psi\colon C_m\to\Aut(P)$.

In the case $\psi$ is trivial, $P$ is not abelian.
Put $K'=\Kbar((t))$.
By Theorem \ref{p-group} proved in section \ref{lemma}, there exists a $P$-extension $F'/K'$ such that the Galois group $P$ is equal to the inertia group $I(F'/K')$ of $F'/K'$, and that $F'/K'$ has a non-integral non-log upper ramification break.
Since $K$ has a totally ramified $C_m$-extension and $p\nmid m$, $\Kbar$ contains a primitive $m$-th root of unity by Lemma \ref{tr-cyclic} proved in Section \ref{lemma}.
Hence $K'$ contains a primitive $m$-th root of unity $\zeta_m$.
Since $f(X)=X^m-t\in K'[X]$ is an Eisenstein polynomial, the extension $M'/K'$ generated by a root of $f$ is a totally ramified extension of degree $m$.
Moreover, since $\zeta_m\in K'$, $M'/K'$ is a cyclic extension by the Kummer theory.
Put $L'=F'M'$.
Since $p\nmid m$, we have $F'\cap M'=K'$ and $L'/K'$ is a $G$-extension.
Since \[P=I(F'/K')=I(L'/K')G(L'/F')/G(L'/F')\] and \[C_m=I(M'/K')=I(L'/K')G(L'/M')/G(L'/M'),\]
the order of $I(L'/K')$ is divisible by $|P|$ and $m$.
Thus $I(L'/K')$ is equal to $G=\Gal(L'/K')$, since $P$ is a $p$-group and $p\nmid m$.
Moreover, since $F'\subset L'$, $L'/K'$ has a non-integral non-log upper ramification break.

In the case $\psi$ is nontrivial, we claim that $L/K$ has a non-integral non-log upper ramification break.
Let $\Phi(P)$ be the Frattini subgroup of $P$ (i.e., the intersection of all maximal subgroups of $P$) and let $M=L^{\Phi(P)}$ be the fixed field of $\Phi(P)$.
Let $f$ be an automorphism of $P$ and let $Q$ be a maximal subgroup of $P$.
Then $f^{-1}(Q)$ is a maximal subgroup of $P$.
Thus we have $f(\Phi(P))\subset f(f^{-1}(Q))=Q$ and hence $f(\Phi(P))\subset\Phi(P)$.
Therefore, $\Phi(P)$ is a normal subgroup of $P$ and $M/K$ is a Galois extension.
Moreover, we have $\Gal(M/K)=\Gal(L/K)/\Gal(L/M)=I(L/K)/\Gal(L/M)=I(M/K)$.
Put $\Pbar=P/\Phi(P)$ and let $\psibar\colon C_m\to\Aut(\Pbar)$ be the homomorphism induced by $\psi$.
Then $\psibar$ is nontrivial by \cite[Chapter 5 Theorem 1.4]{fingro}.
Thus the Galois group $\Gal(M/K)\cong\Pbar\rtimes_{\psibar}C_m$ is not abelian.
Let $M'/K$ be the maximal tamely ramified subextension of $M/K$.
Then the Galois group $\Gal(M/M')=\Pbar$ is abelian by \cite[(23.2)]{asch}.
By applying Theorem \ref{imperfect} to $M/M'/K$, $M/K$ has a non-integral non-log upper ramification break.
Hence $L/K$ has a non-integral non-log upper ramification break.
\end{proof}

\cite[Theorem 6.5]{converse} shows that there exists a $G$-extension of fields of characteristic $0$ which satisfies the conclusion of Theorem \ref{converse1}.
We prove a similar theorem without the assumption that the residue field is perfect.

\begin{theorem}[{cf.\ \cite[Theorem 6.5]{converse}}]\label{converse3}
Let $K$ be a complete discrete valuation field with the residue field $\Kbar$ of residue characteristic $p>2$ and let $L/K$ be a finite Galois extension. 
We assume that the Galois group $G=\Gal(L/K)$ of $L/K$ is non-abelian and that $G$ is equal to the inertia group $I(L/K)$ of $L/K$.
Then there exist a complete discrete valuation field $K'$ of characteristic $0$ with the residue field $\Kbar$ and a $G$-extension $L'/K'$ such that the Galois group $G=\Gal(L'/K')$ is equal to the inertia group $I(L'/K')$ of $L'/K'$, and that $L'/K'$ has a non-integral non-log upper ramification break.
\end{theorem}
\begin{proof}
By Theorem \ref{converse2}, we can assume that there exist a complete discrete valuation field $K_1$ of characteristic $p$ with the residue field $\Kbar$ and a $G$-extension $L_1/K_1$ such that the Galois group $G=\Gal(L_1/K_1)$ is equal to the inertia group $I(L_1/K_1)$ of $L_1/K_1$, and that $L_1/K_1$ has a non-integral non-log upper ramification break.
We take a complete discrete valuation field $K_2$ of characteristic $0$ with the residue field $\Kbar$ whose absolute ramification index $\e(K_2)$ is larger than all the non-log upper ramification breaks of $L_1/K_1$.
By Lemma \ref{category} in the case $j=\e(K_2)$, there exists a $G$-extension $L_2/K_2$ such that the Galois group $G=\Gal(L_2/K_2)$ is equal to the inertia group $I(L_2/K_2)$ of $L_2/K_2$, and that $L_2/K_2$ has a non-integral non-log upper ramification break.
\end{proof}

\section{Lemmas} \label{lemma}

Let $K$ be a complete discrete valuation field with the residue field $\Kbar$, which is not necessarily perfect, of residue characteristic $p>0$, and let $L/K$ be a finite Galois extension of degree $p^n$.
If $u$ is a non-log upper ramification break of $L/K$, then we have $(\Gal(L/K)_{\text{n-log}}^u:\Gal(L/K)_{\text{n-log}}^{u+\epsilon})=p^d$ for a sufficiently small $\epsilon>0$ and some $d\ge1$. 
In this case we say that $u$ is a {\it non-log upper ramification break with multiplicity} $d$.
The non-log upper ramification breaks of $L/K$, counted with multiplicities, form a multiset with the cardinality $n$, which we denote by $\U_{L/K}$.
For example, if $L/K$ is a finite Galois extension of degree $p^3$ and if $L/K$ has the unique non-log upper ramification break $u$, then $u$ has the multiplicity $3$ and we have $\U_{L/K}=\{u,u,u\}$.
If a subextension $M/K$ of $L/K$ is a Galois extension, it follows that $\U_{M/K}\subset\U_{L/K}$, since the non-log upper ramification groups are compatible with quotients by \cite[Remark 3.5]{abbes-saito}.

\begin{lemma} \label{composite}
Let $K$ be a complete discrete valuation field of residue characteristic $p>0$.
For $i=1,2$, let $L_i/K$ be a finite Galois extension of degree $p^{n_i}$.
We assume that the Galois group $\Gal(L_i/K)$ of $L_i/K$ is equal to the (wild) inertia group $I(L_i/K)=P(L_i/K)$ of $L_i/K$ for $i=1,2$.
Let $L=L_1L_2$ be the composite field of $L_1$ and $L_2$, and let $F/K$ be an extension of degree $p^m$ such that $F\subseteq L_1\cap L_2$.
We assume that the cardinality of $\U_{L_1/K}\cap\U_{L_2/K}$ is equal to $m$.
Then $L/K$ is a finite Galois extension of degree $p^{n_1+n_2-m}$ whose Galois group is equal to the (wild) inertia group $I(L/K)=P(L/K)$.
In particular, if $\U_{L_1/K}\cap\U_{L_2/K}=\emptyset$, then $L/K$ is a Galois extension of degree $p^{n_1+n_2}$ whose Galois group is isomorphic to $\Gal(L_1/K)\times\Gal(L_2/K)$ and is equal to the (wild) inertia group $I(L/K)=P(L/K)$.
\end{lemma}
\begin{proof}
Put $N=L_1\cap L_2$.
Then $N/K$ is a finite Galois extension, and we have $\U_{N/K}\subseteq\U_{L_1/K}\cap\U_{L_2/K}$.
Thus we get $[N:K]\le p^m$, and we have $N=F$.
Hence $L/K$ is a finite Galois extension of degree $p^{n_1+n_2-m}$.
Since the cardinality of $\U_{L_1/K}\cup\U_{L_2/K}\subseteq\U_{L/K}$ is equal to $n_1+n_2-m$, we get $\U_{L_1/K}\cup\U_{L_2/K}=\U_{L/K}$.
Since $\Gal(L_i/K)=I(L_i/K)=\Gal(L_i/K)_{\text{n-log}}^1$, every $u\in\U_{L_1/K}\cup\U_{L_2/K}$ satisfies $u\ge1$.
Therefore, we have $\Gal(L/K)=I(L/K)=\Gal(L/K)_{\text{n-log}}^1$.
If $\U_{L_1/K}\cap\U_{L_2/K}=\emptyset$, then we can take $F=K$ and the last assertion also follows.
\end{proof}

For a finite $p$-group $G$, the smallest cardinality of generating sets for $G$ is denoted by $\rank(G)$. 

\begin{lemma}[{cf.\ \cite[Corollary 3.2]{converse}}]\label{embed}
Let $K$ be a complete discrete valuation field of characteristic $p>0$ and let $L/K$ be a finite Galois extension whose Galois group $G=\Gal(L/K)$ is a $p$-group.
We assume that the Galois group $G=\Gal(L/K)$ of $L/K$ is equal to the (wild) inertia group $I(L/K)=P(L/K)$ of $L/K$.
Let $\Gt$ be a finite $p$-group and let $\phi:\Gt\to G$ be a surjective group homomorphism.
Then there exists an extension $M/L$ with the following properties:
\begin{enumerate}[$(1)$]
\item $M/K$ is a Galois extension and there exists an isomorphism of exact sequences
\begin{equation*}
\begin{array}{*{9}c}
1&\Lrightarrow&\Gal(M/L)&\Lrightarrow&\Gal(M/K)
&\Lrightarrow&\Gal(L/K)&\Lrightarrow&1 \\[1mm]
&&\Ldownarrow&&\Ldownarrow&&\parallel&& \\[-1mm]
1&\Lrightarrow&\ker\phi&\Lrightarrow&\Gt
&\overset{\tst\phi}{\Lrightarrow}&G&\Lrightarrow&1.
\end{array}
\end{equation*}
\item The Galois group $G=\Gal(M/K)$ of $M/K$ is equal to the (wild) inertia group $I(M/K)=P(M/K)$ of $M/K$.
\end{enumerate}
\end{lemma}
\begin{proof}
We prove the assertion in a similar way to the proof of \cite[Corollary 3.2]{converse}. 
Let $N=\ker\phi$ and $|N|=p^n$.
If $n\ge2$, there exists a nontrivial normal subgroup $N'$ of $N$ such that $N'$ is contained in the center $Z(\Gt)$ of $\Gt$, since $N\cap Z(\Gt)\neq\{1\}$.
Then $\phi$ is the composition of the canonical surjection $\Gt\to\Gt/N'$ and the homomorphisms $\Gt/N'\to G$ induced by $\phi$, and it suffices to consider these two homomorphism.
Therefore, proceeding by induction on $n$, it suffices to consider the case where $N\cong C_p$.

If the extension $\Gt$ of $G$ by $N$ is split, then we have $\Gt\cong C_p\times G$, since $N\cong C_p$ is contained in $Z(\Gt)$.
By \cite[$\mathrm{III}$ Proposition (2.5)]{fesvos}, there exists a totally ramified $C_p$-extension $F/K$ whose non-log upper ramification break is larger than all the non-log upper ramification breaks of $L/K$.
Then we have $\U_{F/K}\cap\U_{L/K}=\emptyset$, and hence $M=FL$ satisfies $(1)$ and $(2)$ by Lemma \ref{composite}.

If the extension $\Gt$ of $G$ by $N$ is not split, then we claim that $\rank(\Gt)=\rank(G)$.
Since the image of a generating set for $\Gt$ under $\phi$ generates $G$, we have $\rank(\Gt)\ge\rank(G)$.
Let $A$ be a generating set for $G$ such that the cardinality of $A$ is equal to $\rank(G)$.
Let $\At\subset\Gt$ be a subset such that the cardinality of $\At$ is equal to that of $A$ and that $\phi(\At)=A$.
We assume that $\langle\At\rangle\not=\Gt$, and then we have $\ker(\langle\At\rangle\to G)\subsetneq N=C_p$, and hence $\langle\At\rangle\cong G$.
Thus the extension $\Gt$ of $G$ by $N$ is split by $G\to\langle\At\rangle\to\Gt$.
This is a contradiction.
Therefore, we have $\langle\At\rangle=\Gt$ and $\rank(\Gt)=\rank(G)$.

By \cite[$\mathrm{III}$]{witK}, there exists an extension $M/L$ satisfying $(1)$.
We assume that $M/L$ is unramified.
Let $F/K$ be the maximal unramified subextension of $M/K$.
Then $F=M^{I(M/K)}$ is a Galois extension over $K$.
Since $M/L$ is unramified, the residue extension of $M/K$ is not purely inseparable.
Hence we have $[F:K]\ge p$.
Since $F/K$ is unramified and the residue extension of $L/K$ is purely inseparable, we have $F\cap L=K$ and $\Gal(M/K)\cong\Gal(F/K)\times\Gal(L/K)$.
Therefore, we have $\Gal(F/K)\cong C_p$ and $\Gal(M/K)\cong C_p\times G$.
This contradicts the assumption that the extension $\Gt$ of $G$ by $N$ is not split.
Thus the residue extension of $M/L$ is purely inseparable, and hence that of $M/K$ is purely inseparable.
Therefore, $M/K$ satisfies $(2)$.
\end{proof}

\begin{cor}[{cf.\ \cite[Corollary 3.3]{converse}}]\label{G-ext}
Let $K$ be a complete discrete valuation field of characteristic $p>0$ and let $\Gt$ be a finite $p$-group.
Then there exists a $\Gt$-extension $L/K$ whose Galois group $\Gal(L/K)$ is equal to the (wild) inertia group of $L/K$.
\end{cor}
\begin{proof}
Take $G=\{1\}$ in Lemma \ref{embed}.
\end{proof}

We put a partial order on finite $p$-groups by $H\preccurlyeq G$ if $H$ is isomorphic to a quotient of
$G$.  
We say that a finite non-abelian $p$-group $G$ is a {\it minimal non-abelian $p$-group} if $G$ is minimal with respect to $\preccurlyeq$ among finite non-abelian $p$-groups (\cite[Section 4]{converse}).

Suppose $p>2$.
For $n,d\ge1$, finite non-abelian $p$-groups $H(n,d)$ and $A(n,d)$ of order $p^{2n+d}$ are defined in \cite[Section 4]{converse}, and these groups are minimal non-abelian $p$-groups.
Conversely, every minimal non-abelian $p$-group is isomorphic to one of these groups by \cite[Proposition 4.2]{converse}.
For example, $H(1,1)$ is generated by three elements $x$, $y$, and $z$ such that $|x|=|y|=p$, $|z|=p^d$, $xz=zx$, $yz=zy$, and that $[x,y]=z$.
Thus $H(1,1)$ is the Heisenberg $p$-group.

\begin{lemma}[{cf.\ \cite[Proposition 5.2]{converse}}]\label{H(1,1)}
Let $K$ be a complete discrete valuation field of characteristic $p>2$ and let $F/K$ be a totally ramified $C_p$-extension.
Let $b$ be the upper, or equivalently lower, ramification break of $F/K$, and let $a$ be an integer such that $a>b$ and $a\not\equiv0,-b\pmod{p}$.  
Then there exists an extension $L/F$ such that $L/K$ is an $H(1,1)$-extension with the non-log upper ramification breaks $b+1$, $a+1$, and $a+p^{-1}b+1$, and that the Galois group $\Gal(L/K)$ of $L/K$ is equal to the (wild) inertia group $I(L/K)=P(L/K)$ of $L/K$.
Moreover, we have $p\nmid b$ and $L/K$ has a non-integral non-log upper ramification break.
\end{lemma}
\begin{proof}
We construct a desired extension $L/F$ in a similar way to the proof of \cite[Proposition 5.2]{converse} further using Lemma \ref{tan-dom}.
By \cite[$\mathrm{III}$ Proposition (2.3)]{fesvos}, we have $p\nmid b$.
By \cite[$\mathrm{III}$ Proposition (2.4)]{fesvos}, there exists $y\in F$ such that $F=K(y)$, $v_F(y)=-b$, and that $\beta=y^p-y\in K$.
Since $p\nmid b$, we can write $a=bt+ps$ with $t,s\in\Z$.
We can assume $0\le t<p$, and then we have $1\le t\le p-2$ since $a\not\equiv0,-b\pmod{p}$.
Put $\alpha=\pi_K^{-ps}\beta^t$ for a uniformizer $\pi_K$ of $K$.
Then we have $v_K(\alpha)=-a$.
Let $x\in K^{sep}$ satisfy $x^p-x=\alpha$ and put $M=F(x)=K(x,y)$.
By \cite[$\mathrm{III}$ Proposition (2.5)]{fesvos}, $K(x)/K$ is a totally ramified $C_p$-extension with the upper ramification break $a$.
By Lemma \ref{composite}, the Galois group $\Gal(M/K)$ of $M/K$ is equal to the (wild) inertia group of $M/K$, since $\U_{K(x)/K}\cap\U_{F/K}=\emptyset$.
Let $\gamma$ be the inverse of $t+1$ in $\F_p^{\times}\subset K^{\times}$, let $z\in K^{sep}$ satisfy $z^p-z=\alpha y+\gamma\alpha\beta$, and put $L=M(z)=K(x,y,z)$.
By \cite[$\mathrm{III}$ Proposition (2.5)]{fesvos}, $L/M$ is totally ramified, since $v_M(\alpha y+\gamma\alpha\beta)<0$.

We prove that $L/K$ is a Galois extension.
Let $\sigma\in\Gal(F/K)$ satisfy $\sigma(y)=y+1$.
Then for $1\le n<p$, we have \[\sigma^n(\alpha y+\gamma\alpha\beta)=\alpha y+\gamma\alpha\beta+n\alpha=z^p-z+n\alpha.\] 
Hence we get \[N_{F/K}(\alpha y+\gamma\alpha\beta)=\prod_{n=0}^{p-1}(z^p-z+n\alpha),\] where $N_{F/K}\colon F\to K$ is the norm map.
Put \[f(X)=\prod_{n=0}^{p-1}(X^p-X+n\alpha)-N_{F/K}(\alpha y+\gamma\alpha\beta)\in K[X].\]
Then for $0\le n, m<p$, $z+nx+m$ is a root of $f$, since $n\alpha=(nx)^p-nx$.
Since $L$ is the smallest splitting field of a separable polynomial $f(X)(X^p-X-\alpha)(X^p-X-\beta)$ over $K$, $L/K$ is a Galois extension.
Moreover, since the residue field $\Lbar=\Mbar$ is purely inseparable over $\Kbar$, the Galois group $\Gal(L/K)$ of $L/K$ is equal to the (wild) inertia group of $L/K$.
By Lemma \ref{tan-dom}, there exists a complete discrete valuation field $K'$ with the perfect residue field of residue characteristic $p$ and there exists a tangentially dominant morphism $\OO_K\to \OO_{K'}$ of discrete valuation rings. 
Put $F'=FK'=K'(y)$, $M'=MK'=K'(x,y)=F'(x)$, and $L'=LK'=K'(x,y,z)=M'(z)$.
Then the proof of \cite[Proposition 5.2]{converse} shows that $L'/K'$ is a totally ramified $H(1,1)$-extension with the upper ramification breaks $b$, $a$ and $a+p^{-1}b$.
By Lemma \ref{reduce2}, $L/K$ is an $H(1,1)$-extension with the non-log upper ramification breaks $b+1$, $a+1$, and $a+p^{-1}b+1$
\end{proof}

Put $H(0,d)=C_p^d$.
By \cite[Proposition 4.4]{converse}, for $n,d\ge1$, $H(n,d)$ is isomorphic to a quotient group of $H(n-1,d)\times H(1,1)$.
Further, for $n\ge2$ and $d\ge1$, $A(n,d)$ is isomorphic to a quotient group of $A(n-1,d)\times H(1,1)$.
By \cite[Proposition 4.5]{converse}, for $d\ge1$, $A(1,d)$ can be written as a quotient group of $G_d$, where $G_d$ is the subgroup of $H(1,1)\times C_{p^{d+1}}=\langle x,y,z\rangle\times\langle w\rangle$ generated by $xw$, $y$, and $z$.

\begin{lemma}[{cf.\ \cite[Proposition 5.4]{converse}}]\label{minimal}
Let $K$ be a complete discrete valuation field of characteristic $p>2$ and let $n,d\ge1$.
\begin{enumerate}[$(1)$]
\item There is an $H(n,d)$-extension $L/K$ such that the Galois group $\Gal(L/K)$ of $L/K$ is equal to the (wild) inertia group $I(L/K)=P(L/K)$ of $L/K$, and that the largest non-log upper ramification break of $L/K$ is not an integer.
\item There is an $A(n,d)$-extension $L/K$ such that the Galois group $\Gal(L/K)$ of $L/K$ is equal to the (wild)  inertia group $I(L/K)=P(L/K)$ of $L/K$, and that the largest non-log upper ramification break of $L/K$ is not an integer.
\end{enumerate}
\end{lemma}
\begin{proof}
We construct a desired extension $L/K$ in a similar way to \cite[Proposition 5.4]{converse}.

$(1)$ By Corollary \ref{G-ext}, there exists an $H(n-1,d)$-extension $N_1/K$ whose Galois group $\Gal(N_1/K)$ is equal to the (wild) inertia group of $N_1/K$.
Let $v+1$ be the largest non-log upper ramification break of $N_1/K$ and let $a, b$ be integers such that $a>b>v$, $p\nmid b$, and that $a\not\equiv0,-b\pmod{p}$.
By \cite[$\mathrm{III}$ Proposition (2.5)]{fesvos}, there exists a totally ramified $C_p$-extension $F/K$ with the upper ramification break $b$.
By Lemma \ref{H(1,1)}, there exists an extension $N_2/F$ such that $N_2/K$ is an $H(1,1)$-extension with the non-log upper ramification breaks $b+1$, $a+1$, and $a+p^{-1}b+1$, and the Galois group $\Gal(N_2/K)$ of $N_2/K$ is equal to the (wild) inertia group of $N_2/K$.
Put $N=N_1N_2$.
Then $N/K$ is a Galois extension whose Galois group $\Gal(N/K)$ is isomorphic to $\Gal(N_1/K)\times\Gal(N_2/K)\cong H(n-1,d)\times H(1,1)$ and is equal to the (wild) inertia group $I(N/K)=P(N/K)$ of $N/K$ by Lemma \ref{composite}.

By Lemma \ref{tan-dom}, there exists a complete discrete valuation field $K'$ with the perfect residue field of residue characteristic $p$ and there exists a tangentially dominant morphism $\OO_K\to \OO_{K'}$ of discrete valuation rings. 
Put $N_i'=N_iK'$ for $i=1,2$ and $N'=NK'$.
By Lemma \ref{reduce2}, $N_1'/K'$ is a totally ramified $H(n-1,d)$-extension such that $v$ is the largest upper ramification break of $N_1'/K'$, $N_2'/K'$ is a totally ramified $H(1,1)$-extension with the upper ramification breaks $b$, $a$, and $a+p^{-1}b$, and that $N'/K'$ is a totally ramified Galois extension whose Galois group $\Gal(N'/K')$ is isomorphic to $\Gal(N/K)$.
Let $B\subset H(n-1,d)\times H(1,1)$ be a normal subgroup such that $(H(n-1,d)\times H(1,1))/B\cong H(n,d)$.
The proof of \cite[Proposition 5.4 (a)]{converse} shows that the fixed field $L'=N'^B$ of $B$ is a totally ramified $H(n,d)$-extension of $K'$ whose largest upper ramification break is $a+p^{-1}b$, which is not an integer.
Let $L=N^B$ be the fixed field of $B$.
Then $L/K$ is an $H(n,d)$-extension, and we have $L'=LK'$ and $\Gal(L/K)=\Gal(N/K)/B=I(N/K)/B=I(L/K)$.
By Lemma \ref{reduce2}, the non-log upper ramification breaks of $L/K$ are $u+1$ for the upper ramification breaks $u$ of $L'/K'$.
Hence $L/K$ satisfies $(1)$.

$(2)$ If $n\ge2$, we can prove the assertion in a similar way to the case $(1)$.
By Corollary \ref{G-ext}, there exists an $A(n-1,d)$-extension $N_1/K$ whose Galois group $\Gal(N_1/K)$ is equal to the (wild) inertia group of $N_1/K$.
Let $v+1$ be the largest non-log upper ramification break of $N_1/K$ and let $a, b$ be integers such that $a>b>v$, $p\nmid b$, and that $a\not\equiv0,-b\pmod{p}$.
By \cite[$\mathrm{III}$ Proposition (2.5)]{fesvos}, there exists a totally ramified $C_p$-extension $F/K$ with the upper ramification break $b$.
By Lemma \ref{H(1,1)}, there exists an extension $N_2/F$ such that $N_2/K$ is an $H(1,1)$-extension with the non-log upper ramification breaks $b+1$, $a+1$, and $a+p^{-1}b+1$, and the Galois group $\Gal(N_2/K)$ of $N_2/K$ is equal to the (wild) inertia group of $N_2/K$.
Put $N=N_1N_2$.
Then $N/K$ is a Galois extension whose Galois group $\Gal(N/K)$ is isomorphic to $\Gal(N_1/K)\times\Gal(N_2/K)\cong A(n-1,d)\times H(1,1)$ and is equal to the (wild) inertia group $I(N/K)=P(N/K)$ of $N/K$ by Lemma \ref{composite}.

By Lemma \ref{tan-dom}, there exists a complete discrete valuation field $K'$ with the perfect residue field of residue characteristic $p$ and there exists a tangentially dominant morphism $\OO_K\to \OO_{K'}$ of discrete valuation rings. 
Put $N_i'=N_iK'$ for $i=1,2$ and $N'=NK'$.
By Lemma \ref{reduce2}, $N_1'/K'$ is a totally ramified $A(n-1,d)$-extension such that $v$ is the largest upper ramification break of $N_1'/K'$, $N_2'/K'$ is a totally ramified $H(1,1)$-extension with the upper ramification breaks $b$, $a$, and $a+p^{-1}b$, and that $N'/K'$ is a totally ramified Galois extension whose Galois group $\Gal(N'/K')$ is isomorphic to $\Gal(N/K)$.
Let $B\subset A(n-1,d)\times H(1,1)$ be a normal subgroup such that $(A(n-1,d)\times H(1,1))/B\cong A(n,d)$.
The proof of \cite[Proposition 5.4 (b)]{converse} shows that the fixed field $L'=N'^B$ of $B$ is a totally ramified $A(n,d)$-extension of $K'$ whose largest upper ramification break is $a+p^{-1}b$, which is not an integer.
Let $L=N^B$ be the fixed field of $B$.
Then $L/K$ is an $A(n,d)$-extension, and we have $L'=LK'$ and $\Gal(L/K)=\Gal(N/K)/B=I(N/K)/B=I(L/K)$.
By Lemma \ref{reduce2}, the non-log upper ramification breaks of $L/K$ are $u+1$ for the upper ramification breaks $u$ of $L'/K'$.
Hence $L/K$ satisfies $(2)$.

We consider the case $n=1$.
By \cite[$\mathrm{III}$ Proposition (2.5)]{fesvos}, there exists a totally ramified $C_p$-extension $F/K$ whose upper ramification break $b$ is not divisible by $p$.
By Lemma \ref{embed}, there exists an extension $N_1/F$ such that $N_1/K$ is a $C_{p^{d+1}}$-extension whose Galois group $\Gal(N_1/K)$ is equal to the (wild) inertia group of $N_1/K$.
Let $v+1$ be the largest non-log upper ramification break of $N_1/K$ and let $a$ be an integer such that  $a>v$ and $a\not\equiv0,-b\pmod{p}$.
By Lemma \ref{H(1,1)}, there exists an extension $N_2/F$ such that $N_2/K$ is a $H(1,1)$-extension with the non-log upper ramification breaks $b+1$, $a+1$, and $a+p^{-1}b+1$, and the Galois group $\Gal(N_2/K)$ of $N_2/K$ is equal to the (wild) inertia group of $N_2/K$.
Put $N=N_1N_2$.
Since $\U_{N_1/K}\cap\U_{N_2/K}=\{b+1\}$ and $F\subseteq N_1\cap N_2$, $N/K$ is a Galois extension whose Galois group $\Gal(N/K)$ is equal to the (wild) inertia group $I(N/K)=P(N/K)$ of $N/K$ by Lemma \ref{composite}.

By Lemma \ref{tan-dom}, there exists a complete discrete valuation field $K'$ with the perfect residue field of residue characteristic $p$ and there exists a tangentially dominant morphism $\OO_K\to \OO_{K'}$ of discrete valuation rings. 
Put $N_i'=N_iK'$ for $i=1,2$ and $N'=NK'$.
By Lemma \ref{reduce2}, $N_1'/K'$ is a totally ramified $C_{p^{d+1}}$-extension such that $v$ is the largest upper ramification break of $N_1'/K'$, and $N_2'/K'$ is a totally ramified $H(1,1)$-extension with the upper ramification breaks $b$, $a$, and $a+p^{-1}b$, and $N'/K'$ is a totally ramified Galois extension whose Galois group $\Gal(N'/K')$ is isomorphic to $\Gal(N/K)$.
Let $B\subset G_d$ be a normal subgroup such that $G_d/B\cong A(1,d)$
The proof of \cite[Proposition 5.4 (b)]{converse} shows that the Galois group $\Gal(N'/K')$ is isomorphic to the group $G_d$, and that the fixed field $L'=N'^B$ of $B$ is a totally ramified $A(1,d)$-extension of $K'$ whose largest upper ramification break of is $a+p^{-1}b$, which is not an integer.
Let $L=N^B$ be the fixed field of $B$.
Then $L/K$ is an $A(1,d)$-extension, and we have $L'=LK'$ and $\Gal(L/K)=\Gal(N/K)/B=I(N/K)/B=I(L/K)$.
By Lemma \ref{reduce2}, the non-log upper ramification breaks of $L/K$ are $u+1$ for the upper ramification breaks $u$ of $L'/K'$.
Hence $L/K$ satisfies $(2)$.
\end{proof}

\begin{theorem}[{cf.\ \cite[Theorem 5.5]{converse}}]\label{p-group}
Let $K$ be a complete discrete valuation field of characteristic $p>2$ and let $G$ be a finite non-abelian $p$-group. 
Then there is a $G$-extension $L/K$ such that the Galois group $G$ is equal to the (wild) inertia group $I(L/K)=P(L/K)$ of $L/K$, and that $L/K$ has a non-integral non-log upper ramification break.
\end{theorem}
\begin{proof}
We prove the assertion in a similar way to \cite[Theorem 5.5]{converse}. 
By \cite[Proposition 4.2]{converse}, there exists a quotient $\Gbar=G/H$ of $G$ which is isomorphic to either
$H(n,d)$ or $A(n,d)$ for some $n,d\ge1$.  
By Lemma \ref{minimal}, there exists a $\Gbar$-extension $M/K$ such that the Galois group $\Gal(M/K)$ of $M/K$ is equal to the (wild) inertia group of $M/K$, and that the largest non-log upper ramification break of $M/K$ is not an integer.
By Lemma \ref{embed}, there exists an extension $L/M$ such that $L/K$ is a $G$-extension whose Galois group $\Gal(L/K)$ is equal to the (wild) inertia group of $L/K$.  
Hence $L/K$ has a non-integral non-log upper ramification break.
\end{proof}

\begin{lemma} \label{tr-cyclic}
Let $K$ be a complete discrete valuation field with the residue field $\Kbar$ of residue characteristic $p>0$, and let $L/K$ is a totally ramified cyclic extension of degree $m>0$.
We assume that $m$ is not divisible by $p$.
Then $\Kbar$ contains a primitive $m$-th root of unity.
\end{lemma}
\begin{proof}
By \cite[$\mathrm{II}$ PROPOSITION (3.5) (1)]{fesvos}, there exists a uniformizer $\pi_L$ in $L$ such that $\pi_L^m=\pi_K$ for some uniformizer $\pi_K$ in $K$.
Put $f(X)=X^m-\pi_K$.
Then $\pi_L$ is a root of $f$, and $\zeta_m\pi_L$ is a conjugate element of $\pi_L$ over $K$, where $\zeta_m$ is a primitive $m$-th root of unity.
Since $f$ is Eisenstein over $K$, $f$ is irreducible over $K$ and we have $\zeta_m\pi_L\in L$.
Hence $\zeta_m\in\OO_L^{\times}$.
Thus $X^m-1\in\Lbar[X]$ splits completely, and $\Kbar=\Lbar$ contains a primitive $m$-th root of unity.
\end{proof}

Let $K$ be a complete discrete valuation field of residue characteristic $p>0$.
In \cite[Definition 3.4]{abbes-saito}, the ramification filtration $(G_K^a, a\in\Q_{\ge0})$ is defined.
In \cite{hattori}, $\FE_{K}^{<j}$ is defined as the category of finite separable extensions whose ramification is bounded by $j$.
Let $L/K$ be a finite separable extension.
The ramification of $L/K$ is {\it bounded by} $j$ if and only if $G_K^j$ of $K$ acts on $\Hom_{K}(L,K^{sep})$ trivially, where $K^{sep}$ is a separable closure of $K$.
If $L/K$ is a Galois extension, the following conditions are equivalent:
\begin{enumerate}[$(1)$]
\item The ramification of $L/K$ is bounded by $j$.
\item $\Gal(L/K)_{\text{n-log}}^j=1$.
\item The largest non-log upper ramification break of $L/K$ is less than $j$.
\end{enumerate}

\begin{lemma}[{cf.\ \cite[Corollary 4.18]{hattori}}]\label{category}
Let $K_1$ and $K_2$ be complete discrete valuation fields with the residue fields $k_1$ and $k_2$ of residue characteristic $p>0$, respectively.
Let $\e(K_i)$ be the absolute ramification index of $K_i$ if $\ch(K_i)=0$ and an arbitrary positive integer if $\ch(K_i)=p$.
Let $j$ be a positive rational number satisfying $j\le\min_i\e(K_i)$.
Suppose that $k_1$ and $k_2$ are isomorphic to each other.
Then there exists an equivalence of categories \[F_j\colon\FE_{K_1}^{<j}\to\FE_{K_2}^{<j}.\]
In particular, if $L_1/K_1$ is a finite Galois extension whose ramification is bounded by $j$, then $L_2:=F_j(L_1)$ satisfies the following properties:
\begin{enumerate}[$(1)$]
\item $L_2/K_2$ is a finite Galois extension.
\item $\Gal(L_2/K_2)$ is isomorphic to $\Gal(L_1/K_1)$.
\item $u\ge0$ is a non-log upper ramification break of $L_2/K_2$ if and only if $u$ is a non-log upper ramification break of $L_1/K_1$.
\end{enumerate}
\end{lemma}
\begin{proof}
\cite[Corollary 4.18]{hattori} shows the existence of $F_j$ and that $F_u\colon\FE_{K_1}^{<u}\to\FE_{K_2}^{<u}$ is the restriction of $F_j$ for any positive rational number $u\le j$.
We assume that $L_1/K_1$ is a finite Galois extension whose ramification is bounded by $j$.

(1) Let $\tilde{L_2}$ be the Galois closure of $L_2/K_2$.
Since the ramification of $L_2/K_2$ is bounded by $j$, $G_{K_2}^j$ acts on $\Hom_{K_2}(L_2,\tilde{L_2})$ trivially.
Hence $G_{K_2}^j$ acts on $\Hom_{K_2}(\tilde{L_2},\tilde{L_2})$ trivially, and the ramification of $\tilde{L_2}/K_2$ is bounded by $j$.
Thus there exists a finite separable extension $\tilde{L_1}/K_1$ whose ramification is bounded by $j$ such that $F_j(\tilde{L_1})\cong\tilde{L_2}$.
Since $F_j$ is full and faithful, $\Hom_{K_1}(L_1,\tilde{L_1})\to\Hom_{K_2}(L_2,\tilde{L_2})$ and $\Hom_{K_1}(L_1,L_1)\to\Hom_{K_2}(L_2,L_2)$ are bijective.
Since $L_1/K_1$ is a Galois extension, there exists a bijection $\Hom_{K_1}(L_1,L_1)\to\Hom_{K_1}(L_1,\tilde{L_1})$.
Therefore, we have $\Hom_{K_2}(L_2,\tilde{L_2})=\Hom_{K_2}(L_2,L_2)$, and $L_2/K_2$ is a Galois extension.

(2) The map $\Aut_{K_1}(L_1)=\Hom_{K_1}(L_1,L_1)\to\Hom_{K_2}(L_2,L_2)=\Aut_{K_2}(L_2)$ induced by $F_j$ is a group homomorphism.
Since $F_j$ is full and faithful, this homomorphism is bijective.

(3) Since the non-log upper ramification groups are compatible with quotients, it suffices to show that the largest non-log upper ramification break of $L_1/K_1$ is equal to the largest non-log upper ramification break of $L_2/K_2$.
Let $u_i$ be the largest non-log upper ramification break of $L_i/K_i$ for $i=1,2$.
Since the ramification of $L_1/K_1$ is bounded by $j$, we have $u_i<j$.
For any rational number $r$ satisfying $u_i<r\le j$, the ramification of $L_1/K_1$ is bounded by $r$.
Since $F_r$ is the restriction of $F_j$, the ramification of $L_2/K_2$ is bounded by $r$, and we have $u_2<r$.
Therefore, we have $u_2\le u_1$.
We can show that $u_1\le u_2$ in a similar way.
\end{proof}


\begin{thebibliography}{99}
\bibitem{abbes-saito} A. Abbes and T. Saito, Ramification of local fields with imperfect residue fields, {\it Amer.\ J.\ Math.}  124 (2002), no. 5, 879--920.

\bibitem{asch} M. Aschbacher, {\it Finite group theory}, Cambridge University Press, 1986.

\bibitem{converse} G. G. Elder and K. Keating, A converse to the Hasse-Arf theorem, {\it Proc.\ Amer.\ Math.\ Soc.\ Ser.\ B} 10 (2023), 326--340.

\bibitem{feshas} I. B. Fesenko, Hasse-Arf property and abelian extensions, {\it Math.\ Nachr.} 174 (1995), 81--87.

\bibitem{fesvos} I. B. Fesenko and S. V. Vostokov, {\it Local fields and their extensions}, With a foreword by I. R. Shafarevich, Second edition, Transl.\ Math.\ Monogr., 121 American Mathematical Society, Providence, RI, 2002. xii+345 pp.

\bibitem{fingro} D. Gorenstein, {\it Finite groups}, Second edition, Chelsea, New York, 1980.

\bibitem{hattori} S. Hattori, Ramification theory and perfectoid spaces, {\it Compos. Math.} 150 (2014), no. 5, 798--834.

\bibitem{illusie} L. Illusie, {\it Complexe cotangent et d\'eformations I}, Springer Lecture Notes in Math., 239, Springer-Verlag, Berlin, Heidelberg, New York 1971.

\bibitem{graded} T. Saito, Graded quotients of ramification groups of local fields with imperfect residue fields, {\it Amer.\ J.\ Math.} 145 (2023), no. 5, 1389--1464.

\bibitem{ser} J.-P. Serre, {\it Local fields}, Graduate Texts in Mathematics 67, Springer, 1979.

\bibitem{witK} E. Witt, Konstruktion von galoisschen K\"{o}rpen der Charakteristik $p$ zu vorgegebener Gruppe
der Ordnung $p^f$, {\it J. Reine Angew. Math.} 174 (1936), 237--245.

\bibitem{witZ} E. Witt, Zyklische K\"orper und Algebren der Charakteristik $p$ vom Grad $p^n$. Struktur diskret bewerteter perfekter K\"orper mit vollkommenem Restklassenk\"orper der Charakteristik $p$, {\it J. Reine Angew. Math.} 176 (1937), 126--140.

\end{thebibliography}
\end{document}